\documentclass{elsarticle}

\usepackage{amsmath,amssymb,amsthm}
\usepackage{enumerate}
\usepackage{lineno}
\usepackage{tikz}
\usetikzlibrary{cd}

\makeatletter
\def\ps@pprintTitle{%
 \let\@oddhead\@empty
 \let\@evenhead\@empty
 \def\@oddfoot{\centerline{\thepage}}%
 \let\@evenfoot\@oddfoot}
\makeatother

\newtheorem{theorem}{Theorem}[section]
\newtheorem{lemma}[theorem]{Lemma}
\newtheorem{proposition}[theorem]{Proposition}
\newtheorem{corollary}[theorem]{Corollary}

\newdefinition{definition}[theorem]{Definition}
\newdefinition{question}[theorem]{Question}
\newdefinition{example}[theorem]{Example}

\theoremstyle{definition}
\newtheorem*{mainproblem}{Main Problem}

\newcommand{\N}{\mathbb{N}}
\newcommand{\Z}{\mathbb{Z}}
\newcommand{\R}{\mathcal{R}}
\newcommand{\F}{\mathcal{F}}
\newcommand{\res}{\restriction}

\journal{}

\begin{document}

\begin{frontmatter}

\title{Infinite Ramsey-minimal graphs for star forests}

\author{Fawwaz Fakhrurrozi Hadiputra}
\ead{fawwazfh@sci.ui.ac.id}
\author{Valentino Vito}
\ead{valentino.vito@sci.ui.ac.id}

\address{Department of Mathematics, Universitas Indonesia, Depok 16424, Indonesia}

\begin{keyword}
Ramsey-minimal graph \sep infinite graph \sep graph embedding \sep star forest \sep subdivision graph
\MSC[2020] 05C55 \sep 05C63 \sep 05C35 \sep 05C60 \sep 05D10
\end{keyword}

\begin{abstract}
For graphs $F$, $G$, and $H$, we write $F \to (G,H)$ if every red-blue coloring of the edges of $F$ produces a red copy of $G$ or a blue copy of $H$. The graph $F$ is said to be $(G,H)$-minimal if it is subgraph-minimal with respect to this property. The characterization problem for Ramsey-minimal graphs is classically done for finite graphs. In 2021, Barrett and the second author generalized this problem to infinite graphs. They asked which pairs $(G,H)$ admit a Ramsey-minimal graph and which ones do not. We show that any pair of star forests such that at least one of them involves an infinite-star component admits no Ramsey-minimal graph. Also, we construct a Ramsey-minimal graph for a finite star forest versus a subdivision graph. This paper builds upon the results of Burr et al.\ in 1981 on Ramsey-minimal graphs for finite star forests.
\end{abstract}

\end{frontmatter}

\section{Introduction}\label{sec1}

All our graphs are simple and undirected, and we allow uncountable graphs. We start by stating basic definitions. For graphs $F$, $G$ and $H$, we write $F \to (G,H)$ if every red-blue coloring of the edges of $F$ produces a red copy of $G$ or a blue copy of $H$. A red-blue coloring of $F$ is \emph{$(G,H)$-good} if it produces neither a red copy of $G$ nor a blue copy of $H$. If $F \to (G,H)$ and every subgraph $F'$ of $F$ is such that $F' \not\to (G,H)$, then $F$ is \emph{$(G,H)$-minimal}. The collection of all $(G,H)$-minimal graphs is denoted by $\R(G,H)$. A pair $(G,H)$ \emph{admits a Ramsey-minimal graph} if $\R(G,H)$ is nonempty.

If $G$ and $H$ are both finite, then a $(G,H)$-minimal graph exists. Indeed, we can delete finitely many vertices and/or edges of $K_{r(G,H)}$ until it is $(G,H)$-minimal. This observation does not necessarily hold when at least one of $G$ and $H$ is infinite, even though there exists a graph $F$ such that $F \to (G,H)$ in the countable case by the Infinite Ramsey Theorem \cite{ramsey1930} and in general by the Erd\H{o}s--Rado Theorem \cite{erdos1956}. In fact, a pair of countably infinite graphs almost never admits a Ramsey-minimal graph---see Proposition \ref{prop3}. In 2021, Barrett and the second author \cite{barrett2021} introduced a characterization problem for pairs of graphs according to whether or not they admit a minimal graph.

\begin{mainproblem}[\cite{barrett2021}]
Determine which pairs $(G,H)$ admit a Ramsey-minimal graph and which ones do not.
\end{mainproblem}

The primary motivation for posing the main problem is the classic problem of determining whether there are finitely or infinitely many $(G,H)$-minimal graphs. This problem was first introduced in 1976 \cite{burr1976,nesetril1976}, and it was studied for finite graphs in general by Ne{\v{s}}et{\v{r}}il and R{\"o}dl \cite{nesetril1978a,nesetril1978b} and for various classes of graphs by Burr et al.\ \cite{burr1978a,burr1980,burr1981a,burr1982,burr1985}. A result by Burr et al.\ on Ramsey-minimal graphs for finite star forests is relevant to our discussion.

\begin{theorem}[\cite{burr1981b}]\label{burr-et-al}
The pair of star forests $(\bigcup_{i = 1}^s S_{n_i}, \bigcup_{j = 1}^t S_{m_j})$ admits infinitely many Ramsey-minimal graphs for $n_1 \ge \dots \ge n_s \ge 2$ and $m_1 \ge \dots \ge m_t \ge 2$ when $s \ge 2$ or $t \ge 2$.
\end{theorem}

The formulation of the main problem is also motivated by the more recent work of Stein \cite{stein2010,stein2011a,stein2011b} on extremal infinite graph theory. It is a subfield of extremal graph theory that developed after the notion of end degrees was introduced a few years prior \cite{bruhn2007,stein2007}.

Barrett and the second author mainly studied the main problem for pairs $(G,H)$ in general. The following is one of the main results presented in their paper.

\begin{theorem}[\cite{barrett2021}]\label{bar-vit}
Let $G$ and $H$ be graphs, and suppose that $\F$ is a (possibly infinite) collection of graphs such that:
\begin{enumerate}[\normalfont 1.]
    \item For all $F \in \F$, we have $F \to (G,H)$.
    \item For every graph $\Gamma$ with $\Gamma \to (G,H)$, there exists an $F \in \F$ that is contained in $\Gamma$.
\end{enumerate}
The following statements hold:
\begin{enumerate}[\normalfont(i)]
    \item If $F$ is a $(G,H)$-minimal graph, then $F \in \F$ and $F$ is non--self-embeddable.
    \item Suppose that any two different graphs $F_1,F_2 \in \F$ do not contain each other. A graph $F$ is $(G,H)$-minimal if and only if $F \in \F$ and $F$ is non--self-embeddable.
\end{enumerate}
\end{theorem}

This paper instead focuses on pairs $(G,H)$ involving a \emph{star forest}---a union of stars. Our first main result shows that any pair of star forests such that at least one of them involves an infinite-star component admits no Ramsey-minimal graph.

\begin{theorem}\label{mainthm1}
Let $G$ and $H$ be star forests. If at least one of $G$ and $H$ contains an infinite-star component, then no $(G,H)$-minimal graph exists.
\end{theorem}

This theorem is in contrast to Theorem \ref{burr-et-al}, which states that there are infinitely many $(G,H)$-minimal graphs when $G$ and $H$ are disconnected finite star forests with no single-edge components. Loosely speaking, the existence of infinitely many finite minimal graphs hence does not give an indication that a corresponding infinite minimal graph exists.

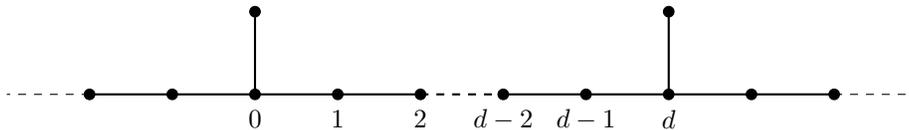
\begin{figure}
    \centering
    \begin{tikzpicture}[x=11mm,y=11mm]
        \draw[thick] (-1,0) -- (3,0);
        \draw[thick] (4,0) -- (8,0);
        \draw[thick] (1,0) -- (1,1);
        \draw[thick] (6,0) -- (6,1);
        \draw[thick,dashed] (3,0) -- (4,0);
        \draw[dashed] (-1,0) -- (-2,0);
        \draw[dashed] (8,0) -- (9,0);
        
        \draw[fill=black] (-1,0) circle (2pt);
        \draw[fill=black] (0,0) circle (2pt);
        \draw[fill=black] (1,0) circle (2pt);
        \draw[fill=black] (2,0) circle (2pt);
        \draw[fill=black] (3,0) circle (2pt);
        \draw[fill=black] (4,0) circle (2pt);
        \draw[fill=black] (5,0) circle (2pt);
        \draw[fill=black] (6,0) circle (2pt);
        \draw[fill=black] (7,0) circle (2pt);
        \draw[fill=black] (8,0) circle (2pt);
        \draw[fill=black] (1,1) circle (2pt);
        \draw[fill=black] (6,1) circle (2pt);
        
        \node at (1,-0.3) {$0$};
        \node at (2,-0.3) {$1$};
        \node at (3,-0.3) {$2$};
        \node at (4,-0.3) {$d-2$};
        \node at (5,-0.3) {$d-1$};
        \node at (6,-0.3) {$d$};
    \end{tikzpicture}
    
    \caption{The graph $F_d$.}
    \label{fig1}
\end{figure}

Similarly, the existence of only finitely many finite minimal graphs does not imply that there are only finitely many corresponding infinite minimal graphs. For $n \in \N$---where $\N$ is the set of positive integers---we denote the $n$-edge star by $S_n$. It is known from \cite{burr1978b} that there are only finitely many $(nS_1,H)$-minimal graphs for $n \in \N$ and $H$ a finite graph. On the other hand, if $\Z$ is the \emph{double ray}---the two-way infinite path---then $(2S_1,\Z \cup S_3)$ admits infinitely many minimal graphs. Indeed, we have $2F_d \in \R(2S_1,\Z \cup S_3)$ for every $d \ge 3$, where $F_d$ is the graph illustrated in Figure \ref{fig1}.

A graph is \emph{leafless} if it contains no vertex of degree one, and it is \emph{non--self-embeddable} if it is not isomorphic to any proper subgraph of itself. Following \cite[p.\ 79]{harary1969}, we denote the \emph{subdivision graph} of $G$ by $S(G)$, which is a graph obtained from $G$ by performing a subdivision on each one of its edges. For example, if $P_n$ denotes the $n$-vertex path, then $S(P_n)=P_{2n-1}$ for $n \in \N$.

For our second main result, we construct a Ramsey-minimal graph for a finite star forest versus the subdivision graph of a connected, leafless, non--self-embeddable graph. In 2020, subdivision graphs were used by Wijaya et al.\ \cite{wijaya2020} to construct new $(nS_1,P_4)$-minimal graphs.

\begin{theorem}\label{mainthm2}
Let $G$ be a connected, leafless, non--self-embeddable graph. For any finite star forest $H$, there exists a $(S(G),H)$-minimal graph.
\end{theorem}

For future investigation, it would be interesting to consider whether every pair of non--self-embeddable graphs admits a minimal graph. If true, this would generalize the observation that a pair of finite graphs always admits a minimal graph, since finite graphs are non--self-embeddable.

\begin{question}\label{ques1}
Is it true that every pair $(G,H)$ of non--self-embeddable graphs admits a Ramsey-minimal graph?
\end{question}

We give an outline of this paper. Section \ref{sec2} discusses self-embeddable graphs and their relevance to the study of Ramsey-minimal graphs. In Section \ref{sec3}, we briefly discuss the Ramsey-minimal properties of $(G,H)$ when $H$ is a union of graphs. Finally, our two main theorems are proved in Sections \ref{sec4} and \ref{sec5}.

\section{Self-embeddable graphs}\label{sec2}

We first provide several preliminary definitions. A \emph{graph homomorphism} $\varphi\colon G \to H$ is a map from $V(G)$ to $V(H)$ such that $\varphi(u)\varphi(v) \in E(H)$ whenever $uv \in E(G)$. A graph homomorphism is an \emph{embedding} if it is an injective map of vertices. Following \cite{bonato2003,bonato2006}, we write $G \le H$ if $G$ embeds into $H$; that is, there exists an embedding $\varphi\colon G \to H$. Unlike in \cite{bonato2003,bonato2006}, however, we do not require that the graph image of $\varphi$ is an induced subgraph of $H$.

A graph $G$ is \emph{self-embeddable} if $G \cong G'$ for some proper subgraph $G'$ of $G$, and the corresponding isomorphism $\varphi\colon G \to G'$ is its \emph{self-embedding}. Examples of self-embeddable graphs include the \emph{ray} $\N$---the one-way infinite path---and a complete graph on infinitely many vertices. On the other hand, finite graphs and the double ray $\Z$ are non--self-embeddable.

Proposition \ref{prop1} provides a necessary and sufficient condition for a graph to be self-embeddable in terms of its components. This proposition is quite similar to \cite[Theorem 2.5]{shekarriz2015} for \emph{self-contained graphs}, the ``induced'' version of self-embeddable graphs.

\begin{proposition}\label{prop1}
A graph $G$ is self-embeddable if and only if at least one of the following statements holds:
\begin{enumerate}[\normalfont (i)]
    \item There exists a self-embeddable component of $G$.
    \item There exists a sequence of distinct components $(C_i)_{i \in \N}$ of $G$ such that $C_1 \le C_2 \le \cdots$.
\end{enumerate}
\end{proposition}

\begin{proof}
The backward direction can be easily proved by defining a suitable self-embedding of $G$ for each of the two cases; it remains to show the forward direction.

Suppose that $G$ has a self-embedding $\varphi$ that embeds $G$ into $G-p$, where $p$ is either a vertex or an edge of $G$, and $G$ contains no self-embeddable component. Let $C_0$ be the component of $G$ containing $p$. We write $v \simeq w$ if the vertices $v$ and $w$ belong to the same component, and we denote $\varphi^k$ as the $k$-fold composition of $\varphi$.

We claim that if $u \in V(C_0)$, then for $0 \le i<j$, we have $\varphi^i(u) \not\simeq \varphi^j(u)$. We use induction on $i$. Let $i=0$, and suppose to the contrary that $u$ and $\varphi^j(u)$, where $j>0$, both belong to $C_0$. If $v \simeq u$, we then have
\[\varphi^j(v) \simeq \varphi^j(u) \simeq u,\]
so $\varphi^j(v) \in V(C_0)$ for every $v \in V(C_0)$. Also, since $\varphi$ embeds $G$ into $G-p$, the map $\varphi^j$ also embeds $G$ into $G-p$. Hence $\varphi^j$ carries $C_0$ into $C_0-p$, which contradicts the non--self-embeddability of $C_0$. Now suppose $i \ge 1$, and suppose to the contrary that $\varphi^i(u)$ and $\varphi^j(u)$, where $j>i$, both belong to the same component $C$. If $v \simeq \varphi^i(u)$, we then have
\[\varphi^{j-i}(v) \simeq \varphi^j(u) \simeq \varphi^i(u),\]
so $\varphi^{j-i}(v) \in V(C)$ for every $v \in V(C)$. We now prove that $\varphi^{j-i}$ carries $C$ into $C-\varphi^i(u)$; this would contradict the non--self-embeddability of $C$. Suppose that $\varphi^{j-i}(v)=\varphi^i(u)$ for some vertex $v$. We have $\varphi^{j-i-1}(v)=\varphi^{i-1}(u)$ by injectivity. By the induction hypothesis, we also have $\varphi^{i-1}(u) \not\simeq \varphi^{j-1}(u)$, so $\varphi^{j-i-1}(v) \not\simeq \varphi^{j-1}(u)$. It follows that $v \not\simeq \varphi^i(u)$, and since $\varphi^i(u) \in V(C)$, we infer that $v \notin V(C)$. Therefore, $\varphi^i(u)$ cannot be the image of a vertex of $C$ under $\varphi^{j-i}$, as desired.

Let $u \in V(C_0)$. Define a sequence $(C_i)_{i \in \N}$ such that $C_i$ is the component containing $\varphi^i(u)$. This sequence consists of pairwise distinct components by the previous claim. It is clear that $\varphi$ carries $C_i$ to $C_{i+1}$, so $C_i \le C_{i+1}$ for $i \in \N$, and we are done.
\end{proof}

Proposition \ref{prop1} implies, as an example, that the union of finite paths is self-embeddable, but the union of finite cycles with different lengths is not. Also, we obtain the following corollary.

\begin{corollary}\label{cor1}
A star forest is self-embeddable if and only if it is infinite.
\end{corollary}

For nonempty graphs $G$, a stronger property than self-embeddability is the property that $G \le G-e$ for all $e \in E(G)$. The ray and an infinite complete graph, for example, enjoy this stronger property. On the other hand, the disjoint union $\N \cup \Z$ is self-embeddable, but does not embed into $\N \cup (\Z-e)$, where $e$ is any edge of $\Z$. Thus $\N \cup \Z$ does not possess this stronger property.

\begin{proposition}\label{prop2}
If $G$ is a nonempty graph such that $G \le G-e$ for all $e \in E(G)$, then no $(G,H)$-minimal graph exists for any graph $H$.
\end{proposition}

\begin{proof}
We will prove that for every graph $F$ such that $F \to (G,H)$, we have $F-e \to (G,H)$ for some $e \in E(F)$. This would show that $(G,H)$ admits no minimal graph.

Let $F$ be a graph, and let $e$ be any one of its edges. Set $F'=F-e$. Suppose that $F' \not\to (G,H)$---there exists a $(G,H)$-good coloring $c'$ of $F'$. We show that $F \not\to (G,H)$. Define a coloring $c$ on $F$ such that $c\res_{E(F')}=c'$ and $e$ is colored red. By this definition, no blue copy of $H$ is produced in $F$. We claim that $c$ does not produce a red copy of $G$ either. Suppose to the contrary that a red copy of $G$, say $\widehat{G}$, is produced in $F$. Since $\widehat{G} \le \widehat{G}-e$, we can choose a red copy of $G$ in $F$ that does not contain $e$; that is, there exists a red copy of $G$ in $F'$. This contradicts the $(G,H)$-goodness of $c'$. As a consequence, $c$ is a $(G,H)$-good coloring of $F$, and thus $F \not\to (G,H)$.
\end{proof}

We note that Proposition \ref{prop2} does not hold for self-embeddable graphs $G$ in general---see Example \ref{ex1}.

If $R$ is the Rado graph, then $R-e$ is also the Rado graph for every $e \in E(R)$ via \cite[Proposition 2(b)]{cameron2013}. As a result, the Rado graph satisfies the hypothesis of Proposition \ref{prop2}. Consequently, by \cite{erdos1963}, the following holds.

\begin{proposition}\label{prop3}
For $H$ a fixed graph, almost all countably infinite graphs $G$ produce a pair $(G,H)$ which admits no Ramsey-minimal graph.
\end{proposition}

\section{Graph unions}\label{sec3}

Before we focus on star forests proper, we provide a quick background on graph unions in general. Consider graphs $G$, $H_1$, and $H_2$; let $F_i \in \R(G,H_i)$ for $i \in \{1,2\}$. Possible candidates for a $(G,H_1 \cup H_2)$-minimal graph include $F_1$, $F_2$, and $F_1 \cup F_2$. 

Although $F_1 \cup F_2 \to (G,H_1 \cup H_2)$, it not necessarily true that $F_1 \cup F_2 \in \R(G,H_1 \cup H_2)$. Indeed, let us take $H_1=H_2=S_1$. For $G$ connected, we have $2G \in \R(G,2S_1)$ provided that $G \in \R(G,S_1)$. This was discussed in \cite{barrett2021} but also follows from Proposition \ref{prop4}. On the other hand, if $G$ is disconnected, we have $3\Z \in \R(2\Z,2S_1)$---not $4\Z$---even though $2\Z \in \R(2\Z,S_1)$.

\begin{proposition}\label{prop4}
Let $G$ and $H$ be nontrivial, connected graphs, and let $n \in \N$. If $F_i \in \R(G,H)$ for $1 \le i \le n$, then
\[\bigcup_{i=1}^n F_i \in \R(G,nH).\]
Consequently, the existence of a $(G,nH)$-minimal graph is assured provided that a $(G,H)$-minimal graph exists.
\end{proposition}

\begin{proof}
The arrowing part is obvious, so we only show the minimality of $\bigcup_{i=1}^n F_i$. It is clear that $F_i \not\to (G,2H)$, since otherwise we would have $F_i \notin \R(G,H)$. Let $e$ be an edge of $F_k$ for some $1 \le k \le n$. Color $F_k-e$ by a $(G,H)$-good coloring and $F_i$, for $i \neq k$, by a $(G,2H)$-good coloring. This coloring on $\left(\bigcup_{i=1}^n F_i\right)-e$ is easily shown to be $(G,nH)$-good from the connectivity of $G$ and $H$. Since $e$ is arbitrary, the proposition is proved.
\end{proof}

In contrast to Proposition \ref{prop4}, the following proposition considers $F_i$ as a candidate for being in $\R(G,H_1 \cup H_2)$. A sufficient condition is provided for a $(G,H_1)$-minimal graph to be $(G,H_1 \cup H_2)$-minimal.

\begin{proposition}\label{prop5}
Let $G$, $H_1$, and $H_2$ be graphs, and let $F \in \R(G,H_1)$. If $F-V(\widehat{H}_1) \to (G,H_2)$ for every $\widehat{H}_1$ a copy of $H_1$ in $F$, then $F \in \R(G,H_1 \cup H_2)$.
\end{proposition}

\begin{proof}
We first prove that $F \to (G,H_1 \cup H_2)$. Suppose $c$ is a coloring on $F$ that produces no red copy of $G$. It follows from $F \to (G,H_1)$ that $c$ produces a blue copy of $H_1$, say $\widehat{H}_1$, in $F$. Let $F'=F-V(\widehat{H}_1)$. Since $F' \to (G,H_2)$ and $F'$ contains no red copy of $G$, there exists a blue copy of $H_2$, say $\widehat{H}_2$, in $F'$. We observe that $\widehat{H}_1$ and $\widehat{H}_2$ are disjoint, so $c$ produces a blue copy of $H_1 \cup H_2$. Hence $F \to (G,H_1 \cup H_2)$. Its minimality follows immediately from the $(G,H_1)$-minimality of $F$.
\end{proof}

\begin{example}\label{ex1}
Let
\begin{align*} 
&G=2S_1, \\
&H_1=\Z, \\
&H_2=\N,\text{ and} \\
&F=2\Z.
\end{align*}
The graph $2\Z$ is $(2S_1,\Z)$-minimal, and $\Z \to (2S_1,\N)$, so we can conclude by Proposition \ref{prop5} that $2\Z \in \R(2S_1,\Z \cup \N)$. This serves as an example of a pair $(G,H)$ involving a self-embeddable graph that admits a minimal graph. We note, however, that no $(S_1,\Z \cup \N)$-minimal graph exists since $\Z \cup \N$ is self-embeddable. Thus it is possible that a $(2G,H)$-minimal graph exists even though no $(G,H)$-minimal graph exists.
\end{example}

\section{Proof of Theorem \ref{mainthm1}}\label{sec4}
We fix star forests $G$ and $H$ such that at least one of them contains a star component on infinitely many vertices. We prove in this section that $(G,H)$ admits no Ramsey-minimal graph.

Suppose that $F \to (G,H)$. Since one of $G$ and $H$ contains a vertex of infinite degree, there exists a vertex $v$ of infinite degree in $F$. We choose an arbitrary edge $e$ at $v$. We prove that $F' \to (G,H)$, where $F'=F-e$. Toward a contradiction, suppose that $F'$ admits a $(G,H)$-good coloring $c'$. Since $\deg(v)$ is infinite, there are two possible cases: $v$ is incident to infinitely many red edges or infinitely many blue edges under the coloring $c'$.

Suppose that $v$ is incident to infinitely many red edges. Define a coloring $c$ on $F$ such that $c\res_{E(F')}=c'$ and $e$ is colored red. This coloring produces no blue copy of $H$, so by $F \to (G,H)$ it produces a red copy of $G$, say $\widehat{G}$, in $F$. There exists a star component $S$ of $\widehat{G}$ that contains $e$ since otherwise, $\widehat{G} \subseteq F'$, which contradicts the $(G,H)$-goodness of $c'$.

If $S$ is infinite, then $F'$ clearly contains a red copy of $G$ by removing $e$ from $\widehat{G}$. On the other hand, let us suppose that $S$ has $n$ vertices. We can pick a red star $S'$ on $n$ vertices that is centered on $v$ but does not contain $e$, since $v$ is incident to infinitely many red edges. The graph $F'$ can then be shown to contain a red copy of $G$ by exchanging $S$ from $\widehat{G}$ for $S'$. In both cases, we obtain a contradiction.

The case when $v$ is incident to infinitely many blue edges can be handled similarly, so our proof of Theorem \ref{mainthm1} is complete.

\section{Subdivision graphs vs.\ star forests}\label{sec5}

\subsection{Bipartite graphs}\label{subsec5.1}

Recall that a graph is \emph{bipartite} if its vertex set can be partitioned into two \emph{parts} such that each part is an independent set. Let $K$ be a bipartite graph with bipartition $\{A,B\}$ such that $\deg(u) \neq \infty$ for all $u \in A$. Before we work on subdivision graphs $S(G)$, we construct for $n \in \N$, a graph $\Gamma(K,A,n)$ such that $\Gamma(K,A,n) \to (K,S_n)$.

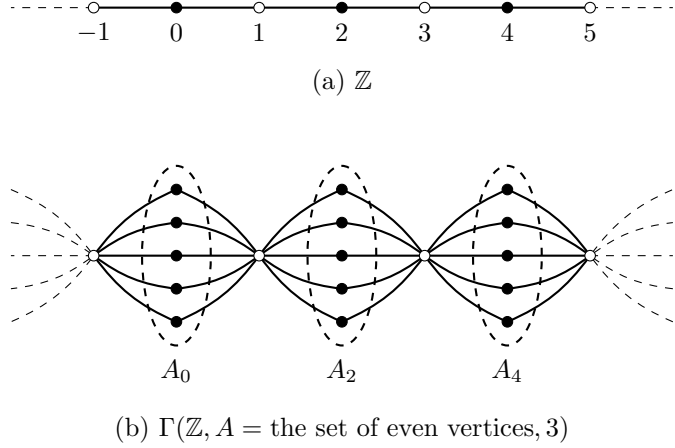
\begin{figure}
    \centering
    \begin{tikzpicture}[x=11mm,y=11mm]
        \draw[thick] (0,0) -- (6,0);
        \draw[dashed] (0,0) -- (-1,0);
        \draw[dashed] (6,0) -- (7,0);
        
        \draw[fill=white] (0,0) circle (2pt);
        \draw[fill=black] (1,0) circle (2pt);
        \draw[fill=white] (2,0) circle (2pt);
        \draw[fill=black] (3,0) circle (2pt);
        \draw[fill=white] (4,0) circle (2pt);
        \draw[fill=black] (5,0) circle (2pt);
        \draw[fill=white] (6,0) circle (2pt);
        
        \node at (0,-0.3) {$-1$};
        \node at (1,-0.3) {$0$};
        \node at (2,-0.3) {$1$};
        \node at (3,-0.3) {$2$};
        \node at (4,-0.3) {$3$};
        \node at (5,-0.3) {$4$};
        \node at (6,-0.3) {$5$};
        
        \node at (3,-0.9) {(a) $\Z$};
        
        \draw[thick] (0,-3) -- (6,-3);
        \draw[thick] (0,-3) to [bend left=15] (1,-2.6) to [bend left=15] (2,-3) to [bend left=15] (3,-2.6) to [bend left=15] (4,-3) to [bend left=15] (5,-2.6) to [bend left=15] (6,-3);
        \draw[thick] (0,-3) to [bend right=15] (1,-3.4) to [bend right=15] (2,-3) to [bend right=15] (3,-3.4) to [bend right=15] (4,-3) to [bend right=15] (5,-3.4) to [bend right=15] (6,-3);
        \draw[thick] (0,-3) to [bend left=15] (1,-2.2) to [bend left=15] (2,-3) to [bend left=15] (3,-2.2) to [bend left=15] (4,-3) to [bend left=15] (5,-2.2) to [bend left=15] (6,-3);
        \draw[thick] (0,-3) to [bend right=15] (1,-3.8) to [bend right=15] (2,-3) to [bend right=15] (3,-3.8) to [bend right=15] (4,-3) to [bend right=15] (5,-3.8) to [bend right=15] (6,-3);
        \draw[dashed] (0,-3) -- (-1,-3);
        \draw[dashed] (6,-3) -- (7,-3);
        \draw[dashed] (-1,-2.6) to [bend left=15] (0,-3) to [bend left=15] (-1,-3.4);
        \draw[dashed] (-1,-2.2) to [bend left=15] (0,-3) to [bend left=15] (-1,-3.8);
        \draw[dashed] (7,-3.4) to [bend left=15] (6,-3) to [bend left=15] (7,-2.6);
        \draw[dashed] (7,-3.8) to [bend left=15] (6,-3) to [bend left=15] (7,-2.2);
        \draw[thick,dashed] (1,-3) ellipse (13pt and 34pt);
        \draw[thick,dashed] (3,-3) ellipse (13pt and 34pt);
        \draw[thick,dashed] (5,-3) ellipse (13pt and 34pt);
        
        \draw[fill=black] (1,-2.6) circle (2pt);
        \draw[fill=black] (3,-2.6) circle (2pt);
        \draw[fill=black] (5,-2.6) circle (2pt);
        \draw[fill=black] (1,-3.4) circle (2pt);
        \draw[fill=black] (3,-3.4) circle (2pt);
        \draw[fill=black] (5,-3.4) circle (2pt);
        \draw[fill=black] (1,-2.2) circle (2pt);
        \draw[fill=black] (3,-2.2) circle (2pt);
        \draw[fill=black] (5,-2.2) circle (2pt);
        \draw[fill=black] (1,-3.8) circle (2pt);
        \draw[fill=black] (3,-3.8) circle (2pt);
        \draw[fill=black] (5,-3.8) circle (2pt);
        \draw[fill=white] (0,-3) circle (2pt);
        \draw[fill=black] (1,-3) circle (2pt);
        \draw[fill=white] (2,-3) circle (2pt);
        \draw[fill=black] (3,-3) circle (2pt);
        \draw[fill=white] (4,-3) circle (2pt);
        \draw[fill=black] (5,-3) circle (2pt);
        \draw[fill=white] (6,-3) circle (2pt);
        
        \node at (1,-4.4) {$A_0$};
        \node at (3,-4.4) {$A_2$};
        \node at (5,-4.4) {$A_4$};
        \node at (3,-5.1) {(b) $\Gamma(\Z,A=\text{the set of even vertices},3)$};
        
    \end{tikzpicture}
    
    \caption{The construction of $\Gamma(K,A,n)$ for $K=\Z$ and $n=3$.}
    \label{fig2}
\end{figure}

We define $\Gamma(K,A,n)$ by adding additional vertices and edges to $K$. For every $u \in A$, we add vertices $u_1,...,u_{m(n-1)}$---each not already in $V(K)$---to $K$, where $m=\deg(u)$. We then insert an edge between each $u_i$ and a vertex $v$ of $K$ if $uv$ exists in $K$. We denote the resulting graph by $\Gamma(K,A,n)$. Also, for each $u \in A$, we define $A_u$ as the set $\{u,u_1,...,u_{m(n-1)}\}$. As a result, $\Gamma(K,A,n)$ admits a bipartition $\{\bigcup_{u \in A} A_u,B\}$. Figure \ref{fig2} shows the result of this construction when $K=\Z$ and $n=3$.

There is a natural projection $\pi\colon \Gamma(K,A,n) \to K$ that is also a homomorphism. It is defined as
\begin{equation}\label{eq1}\pi(v)=\begin{cases}
u, & v \in A_u \text{ for some } u \in A, \\ v, & v \in B.
\end{cases}\end{equation}

\begin{proposition}\label{prop6}
Let $K$ be a bipartite graph with bipartition $\{A,B\}$ such that $\deg(u) \neq \infty$ for all $u \in A$. For $n \in \N$, we have $\Gamma(K,A,n) \to (K,S_n)$. Consequently for $k \in \N$,
\[\bigcup_{i=1}^k \Gamma(K,A,n_i) \to \left(K,\bigcup_{i=1}^k S_{n_i}\right),\]
where $n_1,\dots,n_k \in \N$.
\end{proposition}

\begin{proof}
Suppose that $c$ is a coloring on $\Gamma(K,A,n)$ that produces no blue copy of $S_n$. We prove that $c$ produces a red copy of $K$.

We claim that for all $u \in A$, there exists $v_u \in A_u$ such that $v_u$ is incident to red edges only. By construction, the vertices in $A_u$ share the same neighborhood $N$ of $m$ vertices, and $|A_u|=m(n-1)+1$. If every vertex in $A_u$ is incident to at least one blue edge, then the vertices in $N$ in total are incident to at least $m(n-1)+1$ blue edges. Since $|N|=m$, there exists a vertex in $N$ that is incident to at least $n$ blue edges by the Pigeonhole Principle. This is impossible since $\Gamma(K,A,n)$ does not contain a blue copy of $S_n$. Therefore, $A_u$ must contain a vertex that is incident to only red edges.

Hence we can define an embedding $\varphi\colon K \to \Gamma(K,A,n)$ as
\[\varphi(u)=\begin{cases}
v_u, & u \in A, \\ u, & u \in B.
\end{cases}\]
The graph image of $\varphi$ is a red copy of $K$ in $\Gamma(K,A,n)$, as desired.
\end{proof}

The graph $\Gamma(K,A,n)$ is not necessarily $(K,S_n)$-minimal in general. For example, let us take $K=S_k$ and $A$ as the set of leaf vertices of $S_k$. We have $\Gamma(S_k,A,n)=S_{kn}$, which is not $(S_k,S_n)$-minimal for $k,n \ge 2$ since $S_{k+n-1} \in \R(S_k,S_n)$. However, we potentially have $\Gamma(K,A,n) \in \R(K,S_n)$ when $K=S(G)$ for some graph $G$ as stated in Theorem \ref{mainthm2}.

\subsection{Proof of Theorem \ref{mainthm2}}\label{subsec5.2}

Fix a connected, leafless, non--self-embeddable graph $G$. Building upon Subsection \ref{subsec5.1}, we prove that for $n_1,\dots,n_k \in \N$, we have
\begin{equation}\label{eq2}\bigcup_{i=1}^k \Gamma(S(G),A,n_i) \in \R\left(S(G),\bigcup_{i=1}^k S_{n_i}\right),\end{equation}
where $A$ is taken as the set of vertices of $S(G)$ that subdivide the edges of $G$. We note that $\deg(u)=2$ for all $u \in A$. First, we show that the three properties of $G$ transfer to $S(G)$, and that $S(G)$ is \emph{$C_4$-free}---it contains no $4$-cycles. The following lemma can be verified using elementary means.

\begin{lemma}\label{lem1}
Let $G$ and $H$ be connected, bipartite graph with bipartition $\{A,B\}$ and $\{C,D\}$, respectively. For any isomorphism $\varphi\colon G \to H$, either $\varphi(A)=C$ and $\varphi(B)=D$, or $\varphi(A)=D$ and $\varphi(B)=C$.
\end{lemma}

\begin{proposition}\label{prop7}
If $G$ is a connected, leafless, non--self-embeddable graph, then $S(G)$ is also a connected, leafless, non--self-embeddable graph. In addition, $S(G)$ is $C_4$-free.
\end{proposition}

\begin{proof}
The first two properties obviously transfer, and $S(G)$ is $C_4$-free since $G$ contains no multiple edges. We now prove that $G$ is self-embeddable given that $S(G)$ is self-embeddable.

Suppose that $\varphi$ is a self-embedding of $S(G)$. Let $A$ be the set of vertices of $S(G)$ that subdivide the edges of $G$, and let $B=V(G)$. Since $S(G)$ is connected and bipartite with bipartition $\{A,B\}$, there are by Lemma \ref{lem1} two cases to consider.

\textbf{Case 1:} $\varphi(A) \subseteq A$ and $\varphi(B) \subseteq B$. We claim that $\varphi$, restricted to $V(G)$, gives rise to a self-embedding $\widehat{\varphi}$ of $G$. It is straightforward to show that $\widehat{\varphi}$ is an embedding, so we only prove that there is an edge of $G$ not in the image of $\widehat{\varphi}$. Suppose that $uv$, where $u \in A$ and $v \in B$, is an edge of $S(G)$ not in the image of $\varphi$, and suppose that $u$ subdivides an edge $vw$ of $G$.

We prove that $vw$ is not in the image of $\widehat{\varphi}$. Suppose toward a contradiction that $\widehat{\varphi}(a)=v$ and $\widehat{\varphi}(b)=w$ for two adjacent vertices $a,b \in V(G)$. Let $c$ be the vertex that subdivides $ab$. It is apparent that $\{\varphi(c),v\}$ and $\{\varphi(c),w\}$ are edges of $S(G)$. Also, we cannot have $\varphi(c)=u$ since $uv$ is not in the image of $\varphi$. But then the vertices in the set $\{v,u,w,\varphi(c)\}$ induce a $4$-cycle on $S(G)$, which contradicts the fact that $S(G)$ is $C_4$-free.

\textbf{Case 2:} $\varphi(A) \subseteq B$ and $\varphi(B) \subseteq A$. The map $\varphi^2$ is a self-embedding of $S(G)$ that carries $A$ into $A$, and $B$ into $B$. So by appealing to Case 1, we can obtain a self-embedding of $G$.
\end{proof}

Armed with Proposition \ref{prop7}, we are ready to prove Theorem \ref{mainthm2}. But first, let us provide a straightforward application of the membership statement of (\ref{eq2}) that we will prove later.

\begin{example}\label{ex2}
Choose $G=\Z$ and $H=S_3$. Since $\Z$ is connected, leafless, and non--self-embeddable, and $S(\Z)=\Z$, the graph of Figure \ref{fig2}(b) is $(\Z,S_3)$-minimal by (\ref{eq2}).
\end{example}

\begin{proof}[Proof of Theorem \ref{mainthm2}]
First, suppose
\[H=\bigcup_{i=1}^k S_{n_i}, \text{ where } 1 \le n_1 \le \dots \le n_k.\]
Let $A$ be the set of vertices of $S(G)$ that subdivide the edges of $G$ so that $\deg(u)=2$ for all $u \in A$. Define $\Gamma_i=\Gamma(S(G),A,n_i)$, and let $\Gamma=\bigcup_{i=1}^k \Gamma_i$. Denote the corresponding set to $A_u$ that belongs to $\Gamma_i$ by $A_{u,i}$. We have $|A_{u,i}|=2n_i-1$. If $B_i=V(\Gamma_i) \backslash \bigcup_{u \in A} A_{u,i}$, then $\Gamma_i$ admits a bipartition $\{\bigcup_{u \in A} A_{u,i},B_i\}$.

We prove for each $e \in E(\Gamma)$ that there is a $(S(G),H)$-good coloring of $\Gamma-e$. This, along with Proposition \ref{prop6}, would show that $\Gamma \in \R(S(G),H)$.

\begin{lemma}\label{lem2}
For every $e \in E(\Gamma)$, there exists a coloring $c$ on $\Gamma-e$ such that both of the following statements hold:
\begin{enumerate}[\normalfont(i)]
    \item The coloring $c$ produces no blue copy of $H$.
    \item There exists $u \in A$ such that for $1 \le i \le k$, every vertex in $A_{u,i}$ is incident to exactly one red edge.
\end{enumerate}
\end{lemma}

\begin{proof}
Suppose that $e$ is an edge of some $\Gamma_j$, where $1 \le j \le k$, and that $e$ is at a vertex $v \in A_{u,j}$ for some $u \in A$. We color each edge in every $\Gamma_i$, minus the edge $e$ when $i=j$, by the following rules:

\textbf{Case 1:} $i<j$. Recall that $|A_{u,i}|=2n_i-1$ and that all vertices in $A_{u,i}$ share the same neighborhood $\{a,b\}$. Arbitrarily partition $A_{u,i}$ into sets $S$ and $T$ such that $|S|=n_i$ and $|T|=n_i-1$. Color all the edges in $E(S,a) \cup E(T,b)$ blue, where $E(S,a)$ denotes the set of all edges between the vertex set $S$ and the vertex $a$; this produces two blue stars of sizes $n_i$ and $n_i-1$, respectively. Color the rest of $\Gamma_i$ red.

\textbf{Case 2:} $i=j$. As before, let $a$ and $b$ be the vertices adjacent to each vertex in $A_{u,j}$. Partition $A_{u,j} \backslash v$ into sets $S$ and $T$ both of size $n_j-1$. Similarly to Case 1, we color all the edges in $E(S,a) \cup E(T,b)$ blue. This produces two blue stars of size $n_j-1$. Color the rest of $\Gamma_j$ red.

\textbf{Case 3:} $i>j$. Let $a$ be a vertex adjacent to each vertex in $A_{u,i}$. Color $E(A_{u,i},a)$ blue; this produces a blue star of size $2n_i-1$. As previously, we color the rest of $\Gamma_i$ red.

Denote the preceding coloring scheme by $c$. It is obvious from the preceding construction of $c$ that (ii) holds for our $u \in A$, so it remains to prove that (i) holds.

Let $j'$ be the least positive integer such that $n_{j'}=n_j$. Observe that we only produce blue stars of size at least $n_j$ in Case 3 and, if $j'<j$, in Case 1 also. Every $\Gamma_i$ such that $j'\le i \le k$ and $i \neq j$ contributes exactly one blue star of size at least $n_j$, so exactly $k-j'$ such blue stars are produced in $\Gamma-e$ in total. But $H$ contains $k-j'+1$ stars of size at least $n_j$, so no blue copy of $H$ can be produced in $\Gamma-e$ by the coloring $c$.
\end{proof}

We take the coloring $c$ of Lemma \ref{lem2}. To prove that $c$ is $(S(G),H)$-good, we need to show that $c$ does not produce a red copy of $S(G)$ in $\Gamma-e$.

Suppose to the contrary that there exists an embedding $\xi\colon S(G) \to \Gamma_i$ such that its graph image is a red copy of $S(G)$. Set $\varphi=\pi \circ \xi$, where $\pi\colon \Gamma_i \to S(G)$ is a projection that sends each vertex in $A_{u,i}$ to $u$ and is defined similarly to Eq.\ (\ref{eq1}). We prove that $\varphi$ is a self-embedding of $S(G)$, which would contradict the non--self-embeddability of $S(G)$. For illustration, we provide the following commutative diagram of graph homomorphisms:
\[\begin{tikzcd}[row sep=9mm,column sep=8mm]
    S(G) \arrow[hook,r,"\xi"] \arrow[dr,"\varphi"'] & \Gamma_i \arrow[two heads,d,"\pi"] \\
    & S(G)
\end{tikzcd}\]

Suppose that $\varphi(a)=b$ for some vertices $a$ and $b$ of $S(G)$. If $b \in A$, then the vertex $\xi(a)$ belongs in $A_{b,i}$. Recall that $\deg(a) \ge 2$ since $S(G)$ is leafless. Since the graph image of $\xi$ is red, $\xi(a)$ needs to be incident to at least two red edges as a result. We infer that $b \neq u$, where $u \in A$ is taken from Lemma \ref{lem2}(ii). This shows that the vertex $u$ of Lemma \ref{lem2}(ii) is not in the image of $\varphi$.

Since $S(G)$ is $C_4$-free and $\xi$ is an embedding, there cannot be a $C_4$ in the graph image of $\xi$. We now prove that $\varphi$ is injective. Let $a$ and $b$ be distinct vertices of $S(G)$. Since $a$ and $b$ have degree at least two, the vertices $\xi(a)$ and $\xi(b)$ also have degree at least two. As a result, $\xi(a)$ and $\xi(b)$ cannot both belong in $A_{u,i}$ for some $u \in A$, since that would create a $C_4$ in the graph image of $\xi$. Therefore, $\varphi$ is injective. This completes the proof that $\varphi$ is a self-embedding and finishes our proof of Theorem \ref{mainthm2}.
\end{proof}

\bibliographystyle{elsarticle-harv}
\bibliography{main}

\end{document}